\documentclass[letterpaper]{amsart}

\title{Equipartition of energy in geometric scattering theory}
\author{Dean Baskin}
\thanks{The author is grateful to Maciej Zworski for pointing out the
  connection between the radiation field and equipartition of energy
  in the one-dimensional setting.  This research was supported by NSF
  postdoctoral fellowship DMS-1103436.}
\address{Northwestern University}
\date{April 1, 2013}

\newtheorem{thm}{Theorem}

\newtheorem{prop}[thm]{Proposition}

\theoremstyle{definition}
\newtheorem{defn}{Definition}

\theoremstyle{remark}
\newtheorem{remark}{Remark}

\newcommand{\norm}[2][]{\left\| #2\right\| _{#1}}
\newcommand{\pd}[1][]{\partial_{#1}}
\newcommand{\PD}[1][]{D_{#1}}

\newcommand{\espace}{H_{E}}

\newcommand{\ac}{\operatorname{ac}}
\newcommand{\pp}{\operatorname{pp}}

\newcommand{\grad}{\nabla}
\newcommand{\lap}{\Delta}
\newcommand{\reals}{\mathbb{R}}
\newcommand{\sphere}{\mathbb{S}}

\newcommand{\differential}[1]{\,d#1}
\newcommand{\dg}{\differential{g}}
\renewcommand{\dh}{\differential{h}}
\newcommand{\dr}{\differential{r}}
\newcommand{\ds}{\differential{s}}
\newcommand{\dsigma}{\differential{\sigma}}
\newcommand{\dt}{\differential{t}}
\newcommand{\dx}{\differential{x}}
\newcommand{\dy}{\differential{y}}
\newcommand{\dz}{\differential{z}}
\newcommand{\domega}{\differential{\omega}}

\begin{document}

\maketitle

\begin{abstract}
  In this note, we use an elementary argument to show that the
  existence and unitarity of radiation fields implies asymptotic
  partition of energy for the corresponding wave equation.  This
  argument establishes the equipartition of energy for the wave
  equation on scattering manifolds, asymptotically hyperbolic
  manifolds, asymptotically complex hyperbolic manifolds, and the
  Schwarzschild spacetime. It also establishes equipartition of energy
  for the energy-critical semilinear wave equation on $\reals^{3}$.
\end{abstract}

\section{Introduction}
\label{sec:introduction}

In this note, we use the radiation fields of Friedlander
\cite{Friedlander:1980}, S{\'a} Barreto
\cite{Sa-Barreto:2005,Sa-Barreto:2005}, and Guillarmou and S{\'a}
Barreto \cite{Guillarmou-SB:2008} to demonstrate the equipartition of
energy in the context of geometric scattering theory.  Our method also
provides an alternate proof for equipartition of energy for the
energy-critical semilinear wave equation on $\reals^{3}$ and a proof
of the equipartition of energy on the Schwarzschild exterior.

Asymptotic equipartition of energy for the wave equation on
$\reals^{n}$ was first observed by Brodsky~\cite{Brodsky1967}, whose
proof relied on the Fourier transform.  Using a Paley--Wiener theorem,
Duffin~\cite{Duffin1970} showed equipartition of energy after a finite
time for the wave equation on $\reals^{3}$.  The intervening decades
have seen many proofs of equipartition of energy for equations of
mathematical physics in many contexts.  As there are too many to
provide a comprehensive list, we list a few examples: Dassios and
Grillakis~\cite{Dassios:1983} showed asymptotic equipartition for the
wave equation on an exterior domain, Vega and
Visciglia~\cite{VegaVisciglia2008} proved an integrated version of the
statement for a class of critical semilinear wave equations on
$\reals^{n}$, and Gang and Weinstein~\cite{GangWeinstein2011} proved
an equipartition of mass theorem for nonlinear Schr{\"o}dinger and
Gross--Pitaevskii equations (though the methods discussed in this note
do not apply in the Schr{\"o}dinger setting).
 
In the model setting, we consider an initial value problem of the form
\begin{align}
  \label{eq:IVP}
  \left( \PD[t]^{2} - H \right) u &= 0 \\
  (u(0,z), \pd[t]u (0,z)) &= (u_{0}(z), u_{1}(z)). \notag
\end{align}
Here $\PD = \frac{1}{i}\pd$ and $H$ is a time-independent Hamiltonian
so that the initial value problem~\eqref{eq:IVP} has a conserved
energy $E$.  We suppose that the energy splits (in a time-dependent
manner) into kinetic and potential energies $E_{K}(t)$ and $E_{P}(t)$,
and that this decomposition extends to the level of energy densities
$e_{K}$ and $e_{P}$.  For the wave equation on $\reals^{n}$, these
energy densities are given by $e_{K}= \frac{1}{2}|\pd[t]u|^{2}$ and
$e_{P} = \frac{1}{2}|\grad u |^{2}$ and are given explicitly below in
other contexts.

In a more general Lorentzian setting, this splitting (and, indeed,
conservation of energy) does not occur.  In the example below of the
Schwarzschild spacetime from general relativity, the splitting is
possible because the spacetime is \emph{static}, i.e., $\pd[t]$ is a
Killing vector field orthogonal to the Cauchy hypersurface.

We consider the following geometric contexts (in settings (1)--(3)
below, $n=\dim X$): 
\begin{enumerate}
\item Scattering manifolds $(X,g)$ (in the sense of
  Melrose~\cite{Melrose:1994}), with $H = \lap_{X}$, the Laplacian with
  positive spectrum, with
  \begin{equation*}
    e_{K} = \left| \pd[t]u\right|^{2}, \quad e_{P} = \left| \grad_{g}u\right|_{g}^{2},
  \end{equation*}
\item Asymptotically hyperbolic manifolds $(X,g)$ (in the sense of
  Mazzeo--Melrose \cite{Mazzeo:1987}), with $H = \lap_{X} -
  \frac{(n-1)^{2}}{4}$ restricted to the orthocomplement of the
  eigenfunctions in the pure point spectrum, with
  \begin{equation*}
    e_{K} = \left|\pd[t]u\right|^{2}, \quad  e_{P} = \left| \grad_{g}u \right|^{2} -
    \frac{(n-1)^{2}}{4}|u|^{2},
  \end{equation*}
\item asymptotically complex hyperbolic manifolds $(X,g)$ (in the
  sense of Epstein--Melrose--Mendoza~\cite{Epstein:1991}), with $H =
  \lap_{X} - \frac{n^{2}}{4}$ and again restricted to the
  orthocomplement of the eigenfunctions in the pure point spectrum,
  with
  \begin{equation*}
    e_{K}= \frac{1}{2}\left| \pd[t]u\right|^{2}, \quad e_{P} =
    \frac{1}{2}\left| \grad_{g}u\right|^{2} - \frac{n^{2}}{4}\left| u \right|^{2},
  \end{equation*}
\item the Schwarzschild spacetime (here the Killing vector field
  $\pd[t]$ provides the splitting as well as conservation of energy),
  with
  \begin{equation*}
    e_{K} = \frac{1}{2}\left( 1 - \frac{2M}{r}\right)^{-1}\left|
      \pd[t]u\right|^{2} , \quad e_{P} = \frac{1}{2}\left( 1 -
      \frac{2M}{r}\right)\left| \pd[r]u\right|^{2} + 
    \frac{1}{2r}^{2}\left| \grad_{\omega}u\right|^{2},
  \end{equation*}
\item and the critical semilinear wave equation on $\reals^{3}$, with $H$
  a nonlinear Hamiltonian, and
  \begin{equation*}
    e_{K} = \frac{1}{2}\left| \pd[t] u \right|^{2} , \quad e_{P} =
    \frac{1}{2}\left| \grad u \right|^{2} + \frac{1}{6}\left| u \right|^{6}
  \end{equation*}

\end{enumerate}

The main result of this note is the following theorem:
\begin{thm}
  \label{thm:mainthm}
  If $u$ is a finite energy solution of the initial value
  problem~\eqref{eq:IVP} in one of the contexts above, then $u$
  exhibits asymptotic equipartition of energy, i.e.,
  \begin{equation*}
    \lim_{t\to \infty} E_{K}(t) - E_{P}(t) = 0.
  \end{equation*}
  Here $E_{K}(t)$ and $E_{P}(t)$ are given by integrals of the kinetic
  and potential energy densities, respectively.
\end{thm}

One feature that the contexts above share is the existence and
unitarity of radiation fields.  Radiation fields are restrictions of
(rescaled) solutions of wave equations to null infinity and have
interpretations both as Lax--Phillips translation representations of
wave equations and as generalizations of the Radon transform.
Heuristically, the statement that the radiation field is a unitary
operator is a statement that all energy radiates to null infinity,
i.e., a form of non-quantitative local energy decay.  

We start by considering the model setting of the wave equation on
$\reals^{n}$:
\begin{align*}
  \Box u &= 0 \\
  (u,\pd[t]u)|_{t=0} &= (\phi, \psi).
\end{align*}
Suppose for now that $\phi$ and $\psi$ are smooth and compactly
supported.  For $x\in (0,\infty)$, $\theta \in \sphere^{n-1}$, and
$s \in (-\infty, \infty)$, let us define a new function $v_{+}$ by
\begin{equation*}
  v_{+}(x,s,\theta) = x^{-\frac{n-1}{2}}u\left(s + \frac{1}{x} ,
    \frac{1}{x}\theta\right) .
\end{equation*}
(In other words, $x = |z|^{-1}$ and $s = t - |z|$.)  A relatively
straightforward calculation shows that $v_{+}$ is smooth past $x=0$
and so we can define the \emph{forward radiation field}
$\mathcal{R}_{+}$ by
\begin{equation*}
  \mathcal{R}_{+}(\phi, \psi) (s,\theta) = \pd[s] v_{+}(0,s,\theta).
\end{equation*}
Friedlander~\cite{Friedlander:1980} observed that this radiation field
is a translation representation of the wave group, i.e., it is a
unitary map $\dot{H}^{1}\times L^{2} \to L^{2}(\reals
\times\sphere^{n-1})$ that intertwines wave evolution with
translation.  Moreover, for the flat wave equation, the radiation
field can be written in terms of the Radon transform.  For $X =
\reals^{3}$, the relationship is given by 
\begin{equation*}
  \mathcal{R}_{+}(\phi, \psi) (s,\theta) = -\frac{1}{4\pi} \left(
    R\psi (s,\theta) + \pd[s]\left( s R\phi\right)(s,\theta)\right),
\end{equation*}
where the Radon transform $R$ is given by
\begin{equation*}
  Rf (s,\theta) = \int _{\langle z,\theta\rangle = s}f(z)\dsigma(z) .
\end{equation*}

The existence and unitary of the radiation field have been established
for many other settings in geometric scattering theory, including the
settings enumerated above.  They were first observed by
Friedlander~\cite{Friedlander:1980,Friedlander:2001} in Euclidean and
asymptotically Euclidean spaces.  S{\'a}
Barreto~\cite{Sa-Barreto:2005,Sa-Barreto:2003,Sa-Barreto:2008} proved
support theorems and unitarity for them on asymptotically hyperbolic
and asymptotically Euclidean manifolds, while Guillarmou and S{\'a}
Barreto~\cite{Guillarmou-SB:2008} extended these results to
asymptotically complex hyperbolic manifolds.  S{\'a} Barreto and
Wunsch~\cite{Sa-Barreto:2005a} showed that the radiation field is a
Fourier integral operator with canonical relation given by a ``sojourn
relation.''  In a nonlinear context, S{\'a} Barreto and the
author~\cite{BaskinBarreto2012} proved a support theorem for the
radiation field for the semilinear wave equation on $\reals^{3}$ and
showed it is norm-preserving, while Wang~\cite{Wang:2011} studied the
mapping properties of the radiation field for the Einstein vacuum
equations on perturbations of Minkowski space.  In recent work, Wang
and the author~\cite{Baskin-Wang} showed the existence and unitarity
of radiation fields on the Schwarzschild exterior.

In Section~\ref{sec:one-dimensional-wave} we discuss the motivating
case of the one-dimensional wave equation.  In
Section~\ref{sec:equipartition-energy} we prove a general proposition
implying the equipartition of energy.  In the remaining sections, we
summarize known results for the radiation field in various contexts
and check that they satisfy the conditions of
Section~\ref{sec:equipartition-energy}.

In what follows, $\dg$ denotes the volume form of the relevant metric
$g$.

\section{The one-dimensional wave equation}
\label{sec:one-dimensional-wave}

In this section we discuss the illuminating example of the radiation
field for the one-dimensional wave equation and discuss its connection
with equipartition of energy.

Consider now the one-dimensional wave equation:
\begin{align}
  \label{eq:1D-eqn}
  \left( \PD[t]^{2} - \PD[x]^{2}\right)u &= 0, \\
  (u, \pd[t]u) &= (\phi, \psi) \in C^{\infty}_{c}(\reals) \times
  C^{\infty}_{c}(\reals). \notag
\end{align}
The solution of equation~\eqref{eq:1D-eqn} is given in terms of left-
and right-moving waves:
\begin{align*}
  u(t,x) &= F(x+t) + G(x-t) \\
  F(s) &= \frac{1}{2} \phi(s) + \frac{1}{2} \int_{0}^{s}\psi(r)\dr +C\\
  G(s) &= \frac{1}{2} \phi(s) + \frac{1}{2} \int_{s}^{0}\psi(r)\dr -C 
\end{align*}
The zero-dimensional sphere $\sphere^{0}$ consists of two points,
which we identify as ``left'' ($-1$) and ``right'' ($+1$).  The
forward radiation field associated to the initial data $(\phi, \psi)$
is given by 
\begin{equation*}
  \mathcal{R}_{+}(\phi, \psi)(s, \theta) = \lim_{r\to \infty}\pd[s]u(s +
  r, r\theta) =
  \begin{cases}
    F'(s) & \theta = -1 \\
    G'(s) & \theta = + 1
  \end{cases}
\end{equation*}
In other words, the ``left'' component of the forward radiation field
is the left-moving wave form $F$, while the ``right'' component is the
right-moving wave form $G$.  In particular, we have that
\begin{equation*}
  \mathcal{R}_{+}(\phi, \psi)(s,\theta) =
  \begin{cases}
    \frac{1}{2}\phi'(s) + \frac{1}{2}\psi (s) & \theta = - 1 \\
    \frac{1}{2}\phi'(s) - \frac{1}{2}\psi (s) & \theta = + 1 
  \end{cases}
\end{equation*}
We calculate here the $L^{2}$ norm of $\mathcal{R}_{\pm}$ and its
relationship to the energy of $u$:
\begin{align*}
  \norm[L^{2}(\reals\times
  \sphere^{0})]{\mathcal{R}_{+}(\phi,\psi)}^{2} &=
  \int_{\reals}\left[ \frac{1}{4}\left( \phi'(s) + \psi(s)\right)^{2}
    + \frac{1}{4}\left( \phi'(s) - \psi (s)\right)^{2}\right] \\
  &= \frac{1}{2}\int_{\reals}\left(\left| \phi'(s)\right|^{2} + \left|
    \psi(s)\right|^{2} \right) \ds = E(\phi,\psi)
\end{align*}

It is well-known that the one-dimensional wave equation obeys
equipartition of energy for compactly supported smooth initial data.
Indeed, suppose that both $\phi$ and $\psi$ are smooth and supported
in the ball of radius $R$, so that $F(s)$ and $G(s)$ are both constant
for $|s| \geq R$.  In particular, $G'(s) = F'(s) = 0$ for
$|s|\geq R$.  We then compute
\begin{align*}
  &\frac{1}{2}\int_{\reals}\left( \left| \pd[t]u(t,x)\right|^{2} -
    \left| \pd[x]u(t,x)\right|^{2}\right)\dx  \\
&\quad\quad = \frac{1}{2}\int_{\reals} \left( \left| F'(x+t) - G'(x-t)\right|^{2}
    - \left| F'(x+t)+G'(x-t)\right|^{2}\right)\dx\\
  &\quad\quad= -2\int_{\reals}F'(x+t)G'(x-t)\dx
\end{align*}
For $t \geq R$ and $x\in\reals$, we have that either $|t+x|\geq R$ or
$|t-x|\geq R$ and therefore one of the two factors in the integral
vanishes, so $u$ obeys equipartition of energy.  

In fact, a consequence of the theorem in
Section~\ref{sec:equipartition-energy} is the asymptotic equipartition
of energy for the one-dimensional wave equation.  (This can also be seen
more directly by taking limits in $\dot{H}^{1}\times L^{2}$ above.)

\section{Equipartition of energy}
\label{sec:equipartition-energy}

Suppose now that $s$ is a smooth function of $t$ and $x$ and define
the kinetic and potential parts of the energy for fixed $t$:
\begin{align*}
  E_{K}(\lambda, t) &= \int_{s(t,x) \leq \lambda} e_{K}\dg,\\
  E_{P}(\lambda, t) &= \int _{s(t,x) \leq \lambda}e_{P}\dg.
\end{align*}
We note that the densities $E_{\bullet}(t)$ described in the
introduction are given by $E_{\bullet}(t) = E_{\bullet}(\infty, t)$,
for $\bullet = K,P$.

The results of this paper all follow from the following elementary
proposition.

\begin{prop}
  \label{prop:equipartition}
  Suppose that $u$ satisfies equation~\eqref{eq:IVP} on $\reals
  \times X$ with initial data $(\phi, \psi)$ and there is a function
  $F\in L^{2}(\reals)$ satisfying the following two
  conditions:
  \begin{equation}
    \label{eq:cond1}
    \lim_{T\to \infty} E_{K}(\lambda, T) = \lim _{T\to
      \infty}E_{P}(\lambda, T) = \frac{1}{2}\int
    _{-\infty}^{\lambda}\left| F\right|^{2} \ds ,
  \end{equation}
  for all $\lambda$, and
  \begin{equation}
    \label{eq:cond2}
    \int _{X} \left( e_{K} + e_{P} \right) \dg = \norm[L^{2}(\reals )]{F}^{2}.
  \end{equation}
  Then $u$ exhibits asymptotic equipartition of energy, i.e.,
  \begin{equation*}
    \lim _{T\to \infty} \left( E_{K}(T) - E_{P}(T)\right) = 0.
  \end{equation*}
\end{prop}

\begin{proof}
  Fix $\epsilon > 0$ and take $\lambda _{0}$ such that 
  \begin{equation*}
    \int _{\infty}^{\lambda _{0}} \left|F \right|^{2}\ds \geq \norm[L^{2}(\reals)]{ F}^{2} - \epsilon .
  \end{equation*}
  Condition~\eqref{eq:cond1} allows us to take $T_{0}$ so that
  $E_{K}(\lambda_{0}, T)$ and $E_{P}(\lambda_{0}, T)$ are within
  $\epsilon$ of $\frac{1}{2}\int_{\infty}^{\lambda_{0}}\left|
    F\right|^{2}$ for all $T\geq T_{0}$.
  
  Let $E(t)= \int _{X}\left( e_{K}(t) + e_{P} (t)\right)\dg$.  By
  conservation of energy, $E(t) = E(0)$ for all $t$.
  Condition~\eqref{eq:cond2} implies that $E(t) =
  \norm[L^{2}(\reals)]{F}^{2}$ for
  all $t$.  We now estimate
  \begin{align*}
    \int _{s(T,x) > \lambda_{0}} \left( e_{K}(t) + e_{P}(t)\right)\dg &= E(T) - E_{K}(\lambda_{0},
    T) - E_{P}(\lambda_{0}, T) \\
    &= \norm[L^{2}(\reals)]{F}^{2} -
    E_{K}(\lambda_{0}, T) - E_{P}(\lambda_{0}, T) \\
    &\leq 3\epsilon 
  \end{align*}
   for all $T \geq T_{0}$.   Our choice of $T_{0}$ implies that
   $\left| E_{K}(\lambda_{0}, T) - E_{P}(\lambda_{0}, T)\right| \leq
   2\epsilon$ for all $T \geq T_{0}$.  Putting the two estimates
   together yields that $\left| E_{K} (T) - E_{P}(T)\right| \leq
   5\epsilon$ for all $T\geq T_{0}$, proving the claim.
\end{proof}

\begin{remark}
  \label{rem:local-energy-decay}
  The proof of Proposition~\ref{prop:equipartition} is essentially
  contained in the work of Friedlander~\cite{Friedlander:1980}, who
  observed that such solutions exhibit a non-quantitative form of
  local energy decay.  Indeed, if $\lambda$ and $T_{0}$ are large
  enough, then
  \begin{equation*}
    \int_{s(T,x)> \lambda} \left( e_{K}(t) + e_{P}(t)
    \right) \dg \leq 3\epsilon.
  \end{equation*}
  In particular, after waiting long enough, solutions decay in a
  forward light cone.  Moreover, the argument shows that quantitative
  decay rates for the function $F$ (i.e., for the radiation field
  below) yield quantitative local energy decay statements.
\end{remark}

\section{Applications}
\label{sec:applications}

In this section, we establish conditions~\eqref{eq:cond1} and
\eqref{eq:cond2} for radiation fields in various geometric settings.
In what follows, $X$ is always a compact manifold with boundary $Y$
and $x$ is always a boundary defining function, i.e., $x = 0$ at $Y = \pd
X$ and $dx\neq 0$ at $Y$.

\subsection{Scattering manifolds}
\label{sec:scattering-manifolds}

\begin{defn}
  \label{defn:scattering-mfld}
  A metric $g$ on $X$ is a scattering metric if it is a Riemannian
  metric on the interior of $X$ and, in a collar neighborhood
  $[0,\epsilon)_{x}\times Y_{y}$ of the boundary, $g$ has the form
  \begin{equation*}
    g = \frac{\dx^{2}}{x^{4}} + \frac{h(x,y,\dy)}{x^{2}},
  \end{equation*}
  where $h(x,y,\dy)$ is a smoothly varying (in $x$) family of
  Riemannian metrics on $Y$.
\end{defn}

Scattering metrics are a class of asymptotically conic metrics
introduced by Melrose \cite{Melrose:1994}.  In the case where $Y =
\sphere^{n-1}$ and $h(0,y,\dy)$ is the round metric on the sphere,
they are a class of asymptotically Euclidean metric (in this case
$x=r^{-1}$).

Friedlander~\cite{Friedlander:1980,Friedlander:2001} showed that the
radiation fields for solutions of the wave equation on scattering
manifolds exist and are unitary on the orthogonal complement of the
kernel of the radiation field.  S{\'a} Barreto \cite{Sa-Barreto:2003}
later proved a support theorem for the radiation field and showed that
Friedlander's unitarity theorem could be seen as a consequence of work
of Hassell and Vasy \cite{Hassell:1999}.

Let $u$ be a solution of equation~\eqref{eq:IVP} on a scattering
manifold and let $s_{\pm} (t,x) = t \mp \frac{1}{x}$.  The rescaling $v$,
given by $v_{\pm}(x,s,y) = x^{-\frac{n-1}{2}}u(s\pm \frac{1}{x},
x, y)$ is smooth down to $x=0$ and the radiation field is defined by
\begin{equation*}
  \mathcal{R}_{\pm}(\phi, \psi)(s, y) = \pd[s]v_{\pm}(0, s, y).
\end{equation*}

Consider now, for fixed $t$,
\begin{align*}
  E_{K}(\lambda, t) = \frac{1}{2}\int_{t - \frac{1}{x} \leq
    \lambda}\left| \pd[t]u\right|^{2} \frac{\dx\dh}{x^{n+1}} &=
  \frac{1}{2}\int_{s(t,x)\leq \lambda} \left|
    \pd[s]v(x,s(t,x), y)\right|^{2}\ds \dh .
\end{align*}
The function $v$ is smooth in $x$ and so the integrand converges
uniformly to $\pd[s]v(0,s,y)$ on $s\leq \lambda$ as $t\to \infty$.
The dominated convergence theorem then shows that the first part of
condition~\eqref{eq:cond1} holds for smooth compactly supported data
with $F(s) = \int_{Y}\left| \mathcal{R}_{+}(\phi,\psi)(s,y)\right|^{2}\dh$.
The continuity of the radiation field on $\espace$ shows that holds
for more general data.  A similar calculation shows that the condition
holds for $E_{P}(\lambda, t)$ as well.

The fact that the radiation field is an isometry in this setting shows
that condition~\eqref{eq:cond2} holds on $\espace$, proving the
theorem in the asymptotically Euclidean case.

\subsection{Asymptotically hyperbolic manifolds}
\label{sec:asympt-hyperb-manif}

\begin{defn}
  \label{defn:ah}
  $(X,g)$ is asymptotically hyperbolic if $\overline{g} = x^{2}g$ is a smooth (up to
  the boundary) Riemannian metric on $X$ and $|\dx|_{\overline{g}} =
  1$ at the boundary.
\end{defn}
Graham--Lee~\cite{GrahamLee1991} and Joshi--S{\'a} Barreto~\cite{Joshi:2000} showed that such metrics may
be put into a norma form so that, in a collar neighborhood of the
boundary,
\begin{equation*}
  g = \frac{\dx^{2} + h(x,y,\dy)}{x^{2}},
\end{equation*}
where $h$ is a smoothly varying family of Riemannian metrics on $Y$.
The spectrum of the Laplacian for such a metric was studied by Mazzeo
\cite{Mazzeo:1988,Mazzeo:1991a} and by Mazzeo and
Melrose~\cite{Mazzeo:1987} and consists of an absolutely continuous
spectrum $\sigma_{\operatorname{ac}}(\lap)$ and a finite pure point
spectrum $\sigma_{\operatorname{pp}}(\lap)$.  The spectrum satisfies
\begin{equation*}
  \sigma_{\pp}(\lap) \subset \left( 0 , \frac{(n-1)^{2}}{4}\right) ,
  \quad \sigma_{\ac}(\lap) = \left[ \frac{(n-1)^{2}}{4} , \infty\right) ,
\end{equation*}
giving a decomposition
\begin{equation*}
  L^{2}(X) = L^{2}_{\pp}(X) \oplus L^{2}_{\ac}(X).
\end{equation*}
If $\mathcal{P}_{\ac}: L^{2}(X) \to L^{2}_{\ac}(X)$ is the orthogonal
projector, we let $E_{\ac}(X) = \mathcal{P}_{\ac}(\espace)$.  

We define the operator $H = \lap_{g} - \frac{(n-1)^{2}}{4}$ so that
the bottom of the continuous spectrum moves to $0$.  With this choice
of $H$, the conserved energy is given by
\begin{equation*}
  E(t) = \frac{1}{2}\int _{X} \left( \left|\pd[t]u\right|^{2} + \left| \grad_{g}u\right|^{2} -
    \frac{(n-1)^{2}}{4}\left| u \right|^{2} \right)\dg, 
\end{equation*}
and the energy densities are given by
\begin{align*}
  e_{K} &= \frac{1}{2}\left| \pd[t]u\right|^{2}, \\
  e_{P} &= \frac{1}{2}\left(\left| \grad_{g}u\right|^{2} -
  \frac{(n-1)^{2}}{4}\left| u \right|^{2}\right).
\end{align*}
The conserved energy is positive definite if $u$ is in the image of
$\mathcal{P}_{\ac}$.

S{\'a} Barreto~\cite{Sa-Barreto:2005} showed the existence and
unitarity of the radiation field for initial data in $E_{\ac}$.
Indeed, let $s = t \mp \log x$ and, for a solution $u$ of
equation~\eqref{eq:IVP} (with $H = \lap_{g} - \frac{(n-1)^{2}}{4}$) on
an asymptotically hyperbolic manifold, define the functions
\begin{equation*}
  v_{\pm}(x,s,y) = x^{-n/2}u(s\mp \log x, x,y).
\end{equation*}
S{\'a} Barreto showed that for compactly supported smooth initial data
$v$ is smooth to $x=0$.  As before, the radiation field is given in
terms of $v$:
\begin{equation*}
  \mathcal{R}_{+}(\phi, \psi)(s, y) = \pd[s]v (0,s,y).
\end{equation*}
A similar computation to the one in
section~\ref{sec:scattering-manifolds} shows that
condition~\eqref{eq:cond1} is satisfied, again with $F(s) =
\int_{Y}\left| \mathcal{R}_{+}(s,y)\right|^{2}\dh$.  The unitarity of
the radiation field on $E_{\ac}$ shows that condition~\eqref{eq:cond2}
is satisfied as well, showing that equipartition of energy holds for
initial data in $E_{\ac}$.

One reason for projecting off of the pure point spectrum is that the
eigenfunctions here do not radiate any energy to null infinity.
Moreover, they correspond to zero-energy (and in fact exponentially
growing/decaying) solutions with respect to the energy form above.

\subsection{Asymptotically complex hyperbolic manifolds}
\label{sec:asympt-compl-hyperb}

Guillarmou and S{\'a} Barreto \cite{Guillarmou-SB:2008} showed the
existence and unitarity of the radiation field on asymptotically
complex hyperbolic manifolds.  We refer the reader to their paper for
the relevant definition of asymptotically complex hyperbolic
manifolds.  As with asymptotically real hyperbolic manifolds, the
spectrum splits into an absolutely continuous part and a pure point
part consisting of finitely many eigenvalues.  The radiation field for
equation~\eqref{eq:IVP} with $H= \lap_{g} -\frac{n^{2}}{4}$ is well
defined and unitary on $E_{\ac}$.  The conserved energy is given by
\begin{equation*}
  E (t) = \frac{1}{2}\int _{X}\left( \left| \pd[t]u\right|^{2} +
    \left| \grad_{g}u\right|^{2} - \frac{n^{2}}{4}\left| u\right|^{2}
  \right) \dg ,
\end{equation*}
and computations identical to those in
section~\ref{sec:asympt-hyperb-manif} show that equipartition of
energy holds in this setting.

\subsection{The Schwarzschild spacetime}
\label{sec:schwarzschild}

Recent work by Wang and the author \cite{Baskin-Wang} establishes the
existence and unitarity of the radiation fields on the Schwarzschild
black hole background.  Topologically, the Schwarzschild spacetime is
diffeomorphic to $\reals_{t}\times (2M,\infty)_{r}\times\sphere^{2}$
and is endowed with a Lorentzian metric
\begin{equation*}
  -\left( 1 - \frac{2M}{r}\right)\dt^{2} + \left(
    1-\frac{2M}{r}\right)^{-1}\dr^{2} + r^{2}\domega^{2}.
\end{equation*}
The spatial part of the spacetime has two ends corresponding to
infinity ($r=\infty$) and to the event horizon of the black hole
($r=2M$).  Wang and the author define the radiation field to consist
of two functions: the restriction of the solution $u$ to the event
horizon and the rescaled restriction of the solution to null infinity
and in the process use the pointwise decay of solutions of the wave
equation to show that it is unitary.  

One key feature of the Schwarzschild spacetime is that it is
\emph{static}, i.e., $\pd[t]$ is a Killing vector field that is
orthogonal to the constant $t$ hypersurfaces.  This implies that there
is a conserved energy and that it splits into kinetic and potential parts.

As $\pd[t]$ is a Killing vector field for $r > 2M$, the conserved
energy is given by
\begin{equation*}
  E(t) = \frac{1}{2}\int _{2M}^{\infty}\int_{\sphere^{2}} \left(
    \left( 1 - \frac{2M}{r}\right)^{-1}\left| \pd[t]u\right|^{2} +
    \left( 1 - \frac{2M}{r}\right)\left| \pd[r]u\right|^{2} +
    \frac{1}{r^{2}}\left| \grad_{\omega}u \right|^{2}
  \right) r^{2} \domega \dr, 
\end{equation*}
and so the kinetic and potential energy densities are given by
\begin{align*}
  e_{K} &= \frac{1}{2}\left( 1 - \frac{2M}{r}\right)^{-1}\left|
    \pd[t]u\right|^{2}, \\
  e_{P} &= \frac{1}{2}\left( 1 - \frac{2M}{r}\right) \left| \pd[r]u\right|^{2} +
  \frac{1}{2r^{2}}\left| \grad_{\omega}\right|^{2}.
\end{align*}

Near the event horizon, in coordinates $(\rho = r-2M, \tau = t + r +
2M \log (r-2M))$, the solution is smooth down to $\rho = 0$ and one
half of the radiation field is defined to be
\begin{equation*}
  f_{\mathcal{H}} = 2M\pd[\tau] u |_{\rho = 0}.
\end{equation*}
Near null infinity, in coordinates $(\rho = r^{-1}, \tau = t - r - 2M
\log(r-2M))$, the solution is again smooth down to $\rho = 0$ and the
other half of the radiation field is defined to be
\begin{equation*}
  f_{\mathcal{I}} = \pd[\tau](\rho^{-1} u )|_{\rho = 0}.
\end{equation*}
For finite-energy (with respect to the energy form above) initial
data, the radiation field is unitary, i.e.,
\begin{equation*}
  E(0) = \norm[L^{2}(\reals \times
  \sphere^{2})]{f_{\mathcal{H}}}^{2} +
  \norm[L^{2}(\reals\times\sphere^{2})]{f_{\mathcal{I}}}^{2}. 
\end{equation*}
Proposition~\ref{prop:equipartition} then implies that finite-energy
solutions on the Schwarzschild exterior exhibit equipartition of
energy (on the constant $t$ hypersurfaces).

\subsection{The energy-critical semilinear wave equation on $\reals^3$}
\label{sec:energy-crit-semil}

In recent work \cite{BaskinBarreto2012}, S{\'a} Barreto and the
author define a nonlinear radiation field for the energy-critical
semilinear wave equation on $\reals^{3}$.  Indeed, consider the
following semilinear wave equation on $\reals^{3}$:
\begin{align}
  \label{eq:SLW}
  \pd[t]^{2} - \lap u + |u|^{4}u &= 0, \\
  (u,\pd[t]u)|_{t=0} &= (\phi, \psi). \notag
\end{align}
The energy densities are given by
\begin{align*}
  e_{P} &= \frac{1}{2}\left| \grad u \right| ^{2} + \frac{1}{6}\left|
    u\right|^{2}, \\
  e_{K} &= \frac{1}{2}\left| \pd[t]u\right|^{2},
\end{align*}
and the total energy,
\begin{equation*}
  E = \frac{1}{2}\int _{\reals^{3}}\left( \left| \pd[t]u\right|^{2} +
    \left| \grad u \right| ^{2} \right) \dz + \frac{1}{6} \int
  _{\reals^{3}} \left| u \right|^{6} \dz ,
\end{equation*}
is conserved.

Grillakis \cite{Grillakis:1990} proved global existence and regularity
for solutions of this equation with compactly supported smooth initial
data, while Shatah and Struwe \cite{Shatah:1994} later extended the
global well-posedness result to initial data in $\dot{H}^{1}\times
L^{2}$.  Bahouri and Shatah \cite{Bahouri:1998} showed that the
nonlinear ($u^{6}$) term in the energy tends to $0$ as $t\to\infty$.
Bahouri and G{\'e}rard \cite{Bahouri:1999} extended this result to
show that solutions also exhibit scattering in the energy space.  

The nonlinear radiation field $\mathcal{L}_{+}(\phi, \psi)$ associated
to initial data $(\phi, \psi)$ with finite energy is given just as in
the linear case.  In particular, if $u$ solves equation~\eqref{eq:SLW}
then the limit
\begin{equation*}
  \mathcal{L}_{+}(\phi, \psi)(s,\theta) = \lim_{r\to \infty}\pd[s] u
  (s + r, r\theta)
\end{equation*}
exists because $u \in L^{5}([0,\infty), L^{10}(\reals^{3}))$.
Moreover, the map $(\phi, \psi)\mapsto \mathcal{L}_{+}(\phi, \psi)$ is
norm-preserving in the sense that
\begin{equation*}
  \norm[L^{2}(\reals \times \sphere^{2})]{\mathcal{L}_{+}(\phi, \psi)}^{2}
  = \frac{1}{2}\norm[L^{2}(\reals^{3})]{\grad \phi}^{2} +
  \frac{1}{2} \norm[L^{2}(\reals^{3})]{\psi}^{2} +
  \frac{1}{6}\int_{\reals^{3}}\left| \phi\right|^{6}.
\end{equation*}
The injectivity of the map $(\phi, \psi) \mapsto \mathcal{L}_{+}(\phi,
\psi)$, together with the definition of $\mathcal{L}_{+}(\phi, \psi)$,
shows that the radiation field satisfies properties~\eqref{eq:cond1}
and \eqref{eq:cond2}.  We may thus conclude that the semilinear wave
equation exhibits asymptotic equipartition of energy.  Moreover,
together with the Bahouri and Shatah \cite{Bahouri:1998} result, we
conclude that the critical semilinear wave exhibits asymptotic
equipartition of energy in a more traditional sense, i.e.,
\begin{equation*}
  \lim_{T\to \infty} \int_{\reals^{3}}\left( \left| \pd[t]u(T)\right|^{2} -
  \left| \grad u (T)\right|^{2}\right) \dz = 0.
\end{equation*}
This provides an alternative proof of a result similar to the
integrated equipartition result of
Vega--Visciglia~\cite{VegaVisciglia2008}.

\bibliographystyle{alpha}
\bibliography{papers}

\end{document}